\documentclass{amsart}

\newtheorem{theorem}{Theorem}[section]

\theoremstyle{definition}
\newtheorem{definition}[theorem]{Definition}

\newtheorem{corollary}[theorem]{Corollary}
\newtheorem{proposition}[theorem]{Proposition}
\theoremstyle{remark}
\newtheorem{remark}[theorem]{Remark}

\numberwithin{equation}{section}

\begin{document}

\title{HARMONIC COMPLEX STRUCTURES}
\footnote{Some parts of this paper was published with same title in
Chinese Ann. Math. Ser. A 30 (2009), no. 6, 761--764 }%

\author{}
\address{}
\curraddr{}
\thanks{}

\author{WAN JIANMING}
\address{Center of Mathematical Sciences
Zhejiang University Hangzhou,Zhejiang,310027, China}
\email{wanj\_m@yahoo.com.cn}
\thanks{}


\date{}

\dedicatory{}


\begin{abstract}
In this paper, we introduce a new concept so called harmonic complex
structure by using harmonic theory for vector bundle-valued
differential forms. It is a new structure intermediates between
complex structure and K\"{a}hler structure. From differential
geometric viewpoint, it is a natural generalization of K\"{a}hler
structure.

\end{abstract}

\maketitle

\section*{}

\specialsection*{}
%

\section{INTRODUCTION}
 The harmonic theories for vector bundle-valued
differential forms play an important role in harmonic maps and
Yang-Mills theories. The main idea is to construct some harmonic
theory for vector bundle-valued differential forms such that the
harmonic maps or Yang-Mills fields are exact the harmonic forms.
Motivated by these ideas, we find, since the almost complex
structure is a tangent bundle-valued 1-form, it is very natural to
use harmonic theory of vector bundle-valued differential forms to
study it. Though this is a very classical thing, until nowadays, we
do not find it elsewhere.

Let $M$ be an almost complex manifold. $J$ is an almost complex
structure of $M$, i.e. a smooth section of $\Gamma(T^{*}M\otimes
TM)$ such that $J^{2}=-1$ as a endomorphism $J: TM\longrightarrow
TM$. Given a Riemannian metric, we can define the Hodge-Laplace
operator $\Delta$ on $T^{*}M\otimes TM$. The harmonic complex
structure is defined by
$$\Delta J=0.$$ When $M$ is compact, an almost complex
structure $J$ is harmonic complex if and only if for all $X,
Y\in\Gamma(TM)$ $\nabla J(X, Y)=\nabla J(Y, X)$  and $ Trace\nabla
J=0$. Apparently, a K\"{a}hler structure must be harmonic complex.
The symmetry also implies that $J$ is integrable (Proposition 2.4).
This is the reason that we do not call it harmonic almost complex
structure. In fact, Wood [2] defined the harmonic almost complex
structure from the viewpoint of energy variation of fiber bundles.
But his definition do not imply the integrability. Similar to
harmonic maps, we can define the energy of almost complex
structures. We also get the Bochner type formula for harmonic
complex structures and give several applications.
Particularly, we
prove that $S^{6}$ with standard metric can not admit any harmonic
complex structure.

 Just as minimal submanifold
is the generalization of totally geodesic submanifold, harmonic
complex structure is the natural generalization of  K\"{a}hler
structure. We hope that harmonic complex structures can give some
new insights about K\"{a}hler structures and complex structures. We
also hope that it can arouse people's interesting as a new geometric
structure.

\textbf{Acknowledgement } I would like to thank Yin Fangliang for
introducing [2] to my attention.

\section{HARMONIC COMPLEX STRUCTURES}
This section will be separated in two parts. In first part we will
give a brief introduction of harmonic theories for tangent
bundle-valued differential forms. In the second part we discuss the
harmonic complex structures.

\subsection{Review of harmonic theories for
tangent bundle-valued differential forms}

We only work with tangent bundle-valued differential forms. For
general vector bundles and more details, we recommend the book of
Xin [3].

Let $(M, g)$ be a Riemannian manifold. $\nabla$ is the Levi-Civita
connection associated with $g$. Let $TM$ (resp. $T^{*}M$) denote the
tangent (resp. cotangent) bundle of $M$. We denote by
$\Gamma(\wedge^{p}T^{*}M\otimes TM)$ the set of tangent
bundle-valued $p$-forms over $M$. The Levi-Civita connection
$\nabla$ can be extended canonically to
$\Gamma(\wedge^{p}T^{*}M\otimes TM)$ by
$$(\nabla_{X}\omega)(X_{1}, ..., X_{p})$$ $$=\nabla_{X}(\omega(X_{1}, ..., X_{p}))-\sum_{j}
(\omega(X_{1}, ..., \nabla_{X}X_{j}, ..., X_{p})),$$ for any
$\omega\in\Gamma(\wedge^{p}T^{*}M\otimes TM)$ and $X, X_{1}, ...,
X_{p} \in\Gamma(TM)$.

We can define the differential operator $d :
\Gamma(\wedge^{p}T^{*}M\otimes TM) \longrightarrow
\Gamma(\wedge^{p+1}T^{*}M\otimes TM)$. For any
$\omega\in\Gamma(\wedge^{p}T^{*}M\otimes TM)$ and $X_{0}, X_{1},
..., X_{p} \in\Gamma(TM)$, $$d\omega(X_{0}, ...,
X_{p})=(-1)^{k}(\nabla_{X_{k}}\omega)(X_{0}, ..., \hat{X}_{k}, ...,
X_{p}),$$ where $\hat{X}_{k}$ denotes removing ${X}_{k}$. The
co-differential operator $\delta : \Gamma(\wedge^{p}T^{*}M\otimes
TM) \longrightarrow \Gamma(\wedge^{p-1}T^{*}M\otimes TM)$ is defined
by
$$\delta\omega(X_{1}, ..., X_{p-1})=-(\nabla_{e_{i}}\omega)(e_{i},
X_{1}, ..., X_{p-1}),$$ where $\{e_{i}\}$ is the local orthonormal
frame field.
\begin{remark}
It is easy to check that the above differential operator does not
satisfy $d^{2}=0$. So there is no Hodge theory for vector
bundle-valued differential forms.
\end{remark}

Now we can define Hodge-Laplace operator $$\Delta=d\delta+\delta
d.$$ If $\Delta\omega=0$, we say that $\omega$ is harmonic. Similar
to differential forms, we also have Bochner techniques for vector
bundle-valued differential forms, which play an important role in
harmonic map theories (see [3]). If $M$ is compact, from [3], we
know that $\Delta\omega=0$ if and only if $d\omega=0$ and
$\delta\omega=0$.
\subsection{Harmonic complex structures}

Let $M$ be an compact almost complex manifold. $J$ is the almost
complex structure of $M$.

\begin{definition}
We call that $J$ is a harmonic complex structure if $\Delta J=0$.
\end{definition}
\begin{remark}
By the definition of $d$ and $\delta$, $\Delta J=0$ if and only if
$\nabla J(X, Y)=\nabla J(Y, X)$ for all $X, Y\in\Gamma(TM)$ and $
Trace\nabla J=0$. Recall that a K\"{a}hler structure means a
hermitian complex structure $J$ such that $\nabla J=0$ (see [1]). We
know immediately that a K\"{a}hler structure must be a harmonic
complex structure.
\end{remark}

\begin{proposition}
A harmonic complex structure must be a complex structure.
\end{proposition}
\begin{proof}
From the definition of $dJ$, for any $X, Y\in\Gamma(TM)$, we have
\begin{eqnarray*}
dJ(X, Y) & = & (\nabla_{X}J)(Y)-(\nabla_{Y}J)(X)\\
         & = & \nabla_{X}JY-J(\nabla_{X}Y)-\nabla_{Y}JX+J(\nabla_{Y}X)\\
         & = & [X, JY]+\nabla_{JY}X-[Y, JX]-\nabla_{JX}Y-J[X,Y]\\
\end{eqnarray*}
and
\begin{eqnarray*}
dJ(JX, JY) & = & (\nabla_{JX}J)(JY)-(\nabla_{JY}J)(JX)\\
           & = & -\nabla_{JX}Y-J(\nabla_{JX}JY)+\nabla_{JY}X+J(\nabla_{JY}JX)\\
           & = & \nabla_{JY}X-\nabla_{JX}Y-J[JX, JY].\\
\end{eqnarray*}
Hence we have $$dJ(X,Y)-dJ(JX, JY)=N(J)(X, Y),$$ where $N$ is the
Nijenhuis tensor
$$N(J)(X, Y)=[JX, Y]+[X, JY]+J[JX, JY]-J[X, Y].$$ Recall that $\Delta
J=0$ implies $dJ=0$, by Newlander-Nirenberg theorem (see [1]), we
get the proposition.

\end{proof}

 If a Riemannian metric is $J$ invariant, we call it
almost-hermitian. A nearly K\"{a}hler manifold is an
almost-hermitian manifold such that $(\nabla_{X}J)(X)=0$ for any
$X\in\Gamma(TM)$. The following proposition shows that harmonic
complex structures measure the difference between nearly K\"{a}hler
and K\"{a}hler exactly.

\begin{proposition}
If $M$ is almost-hermitian corresponding $J$. Then $M$ is K\"{a}hler
if and only if $J$ is a harmonic complex structure and $M$ is nearly
K\"{a}hler corresponding $J$.
\end{proposition}
\begin{proof}
The nearly  K\"{a}hler implies $(\nabla_{X}J)(Y)=-(\nabla_{Y}J)(X)$.
Combining with $dJ(X, Y)=0$, one gets $\nabla J=0$. This shows that
$M$ is  K\"{a}hler. The contrary is trivial.
\end{proof}

\section{BOCHNER TYPE FORMULA FOR HARMONIC COMPLEX STRUCTURES AND APPLICATIONS}

Similar to the differential forms, we also have the following
Weitzenb\"{o}ck formula for tangent bundle-valued differential
forms.
\begin{proposition}
([3]) For any tangent bundle-valued $p$-form $\omega$, we have
$$\Delta\omega=-\nabla^{2}\omega+S,$$ where
$\nabla^{2}\omega=\nabla_{e_{i}}\nabla_{e_{i}}\omega-\nabla_{\nabla_{e_{i}e_{i}}}\omega$
and $$S(X_{1}, ..., X_{p})=(-1)^{k}(R(e_{i}, X_{k})\omega)(e_{i},
X_{1}, ..., \hat{X_{k}}, ..., X_{p}),$$ for any $X_{1}, ...,
X_{p}\in\Gamma(TM)$. $R$ is the curvature tensor $R(X,
Y)=-\nabla_{X}\nabla_{Y}+\nabla_{Y}\nabla_{X}+\nabla_{[X, Y]}$ and
$\{e_{i}\}$ is the local orthonormal frame field.
\end{proposition}

Now we use proposition 3.1 to deduce Bochner type formula for
harmonic complex structures. Let $\{e_{i}\}$ is the local
orthonormal frame field. We can define the energy density of an
almost complex structure $J$ by
$$e(J)=\frac{1}{2}<Je_{i}, Je_{i}>.$$ Obviously
$e(J)$ is independent on the choosing of $\{e_{i}\}$.
If $J$ is a harmonic complex
structure, we have following Bochner type formula.
\begin{proposition}
$\Delta e(J)=|\nabla J|^{2}-<R(e_{i}, e_{j})Je_{i},
Je_{j}>+<JR(e_{i}, e_{j})e_{i}, Je_{j}>$, where $|\nabla
J|^{2}=|(\nabla _{e_{i}}J)(e_{j})|^{2}$.
\end{proposition}

\begin{proof}
 First we have
\begin{eqnarray*}
-S(X) & = & (R(e_{i}, X)J)e_{i}\\
      & = &((-\nabla_{e_{i}}\nabla_{X}+\nabla_{X}\nabla_{e_{i}}+\nabla_{[e_{i},X]})J)e_{i}\\
      & = & -\nabla_{e_{i}}((\nabla_{X}J)e_{i})+(\nabla_{X}J)\nabla_{e_{i}}e_{i}+\nabla_{X}((\nabla_{e_{i}}J)e_{i})\\
      & & -(\nabla_{e_{i}}J)\nabla_{X}e_{i}+ \nabla_{[e_{i},X]}Je_{i}-J\nabla_{[e_{i},X]}e_{i}\\
      & = & -\nabla_{e_{i}}(\nabla_{X}Je_{i}-J(\nabla_{X}e_{i}))+
      \nabla_{X}(J\nabla_{e_{i}}e_{i})-J(\nabla_{X}\nabla_{e_{i}}e_{i})\\
      & & +
      \nabla_{X}(\nabla_{e_{i}}Je_{i}-J(\nabla_{e_{i}}e_{i}))
      -\nabla_{e_{i}}(J\nabla_{X}e_{i})+J(\nabla_{e_{i}}\nabla_{X}e_{i})\\
      & & +\nabla_{[e_{i}, X]}Je_{i}-J\nabla_{[e_{i}, X]}e_{i}\\
      & = & R(e_{i}, X)Je_{i}-JR(e_{i}, X)e_{i}.\\
\end{eqnarray*}
 Then
$$<S, J>=-<R(e_{i}, e_{j})Je_{i}, Je_{j}>+<J(R(e_{i}, e_{j})e_{i}),Je_{j}>,$$
and
\begin{eqnarray*}
<\nabla^{2}J, J> & = & <\nabla_{e_{i}}\nabla_{e_{i}}J, J>=e_{i}<\nabla_{e_{i}}J, J>-<\nabla_{e_{i}}J, \nabla_{e_{i}}J>\\
                 & = & \frac{1}{2}e_{i}e_{i}<J, J>-|\nabla J|^{2}=\Delta e(J)-|\nabla J|^{2},\\
 \end{eqnarray*}
here we choose the normal frame field (i.e.
$\nabla_{e_{i}}e_{j}|_{p}=0$ for a fixed point $p$). By
Weizenb\"{o}ck formula,
$$0=<\Delta J, J>=-<\nabla^{2}J, J>+<S, J>,$$ we get the formula.
\end{proof}

We hope that proposition 3.2 contributes to studying of complex
structures and K\"{a}hler structures.

\begin{remark}
If $J$ is only an almost complex structure, from the last step of
proof of proposition 3.2, we have $$\Delta e(J)+<\Delta J,
J>=|\nabla J|^{2}-<R(e_{i}, e_{j})Je_{i}, Je_{j}>+<JR(e_{i},
e_{j})e_{i}, Je_{j}>.$$
\end{remark}

\begin{corollary}
For a compact almost complex manifold, $J$ is a harmonic complex
structure if and only if $\int_{M}(|\nabla J|^{2}-<R(e_{i},
e_{j})Je_{i}, Je_{j}>+<JR(e_{i}, e_{j})e_{i}, Je_{j}>)dv=0$.
\end{corollary}

\begin{corollary}
If $M$ admits a Hermitian harmonic complex structure, then the scale
curvature $\leq <R(e_{i}, e_{j})Je_{i}, Je_{j}>$. The equal holds if
and only if $M$ is K\"{a}hler.
\end{corollary}

Though we do not know whether $S^{6}$ has a complex structure, as an
application of proposition 3.2, we have
\begin{theorem}
$S^{6}$ with standard metric can not admit any harmonic complex
structure.
\end{theorem}
\begin{proof}
If on the contrary, $J$ is a harmonic complex structure. Locally, we
can write $Je_{i}=J^{k}_{i}e_{k}$. Under the standard metric, the
curvature on $S^{6}$ can be written as
$R_{ijkm}=\delta_{ik}\delta_{jm}-\delta_{jk}\delta_{im}$. Then
\begin{eqnarray*}
<JR(e_{i}, e_{j})e_{i}, Je_{j}> & = & <J(R^{k}_{iji}e_{k}),Je_{j}>\\
                                & = & R^{k}_{iji}<Je_{k},Je_{j}>=R_{ijik}J^{m}_{k}J^{m}_{j}\\
                                & = & R_{ijij}(J^{m}_{j})^{2}=\sum_{m}(J^{m}_{j})^{2}\\
 \end{eqnarray*}
and
\begin{eqnarray*}
<R(e_{i}, e_{j})Je_{i}, Je_{j}> & = & <J^{k}_{i}R^{l}_{ijk}e_{l},J^{m}_{j}e_{m}> \\
                                & = & J^{k}_{i}J^{m}_{j}R_{ijkm}=J^{i}_{i}J^{j}_{j}-J^{j}_{i}J^{i}_{j}.\\
 \end{eqnarray*}
Since $J$ is almost complex structure, we have
$traceJ=\sum_{i}J^{i}_{i}=0$ and
$\delta^{j}_{i}=|-\delta^{j}_{i}|=|\sum_{k}J^{k}_{i}J^{j}_{k}|\leq
\sum_{k}\frac{(J^{k}_{i})^{2}+(J^{j}_{k})^{2}}{2}$. So $$\sum_{i,
j}<JR(e_{i}, e_{j})e_{i}, Je_{j}>=6\sum_{i, j}(J^{j}_{i})^{2}>6$$
and
$$\sum_{i, j}<R(e_{i}, e_{j})Je_{i},
Je_{j}>=(\sum_{i}J^{i}_{i})^{2}+6=6.$$ Which is a contradiction to
corollary 3.4.

\end{proof}

\begin{corollary}
 The standard metric on $S^{6}$ with small perturbation
still can not admit any harmonic complex structure.
\end{corollary}

\section{Some properties of Trace}
In this section we study the trace of $A\in\Gamma(TM\otimes
TM^{*})=\Gamma(Hom(TM, TM))$.

From the proof of proposition 3.2, we know $$-S(X)=R(e_{i},
X)Ae_{i}-AR(e_{i}, X)e_{i}.$$ And\begin{eqnarray*}
<(\nabla^{2}A)(e_{i}),e_{i}> & = & <(\nabla_{e_{k}}\nabla_{e_{k}}A)e_{i},e_{i}>\\
                             & = & <\nabla_{e_{k}}\nabla_{e_{k}}(Ae_{i}),e_{i}>\\
                             & = & e_{k}<\nabla_{e_{k}}(Ae_{i}), e_{i}>\\
                             & = &e_{k}e_{k}<Ae_{i},e_{i}>\\
                             & = &\Delta Trace(A),\\
 \end{eqnarray*}
where choosing the normal frame field. Recall that $$Trace \Delta
A=<(\Delta A)e_{i}, e_{i}>.$$

Hence we have
$$Trace \Delta A+\Delta Trace(A)=<AR(e_{i},
e_{j})e_{i},e_{j}> -<R(e_{i}, e_{j})Ae_{i}, e_{j}>.$$ The curvature
term is
\begin{eqnarray*}
<AR(e_{i}, e_{j})e_{i},e_{j}> -<R(e_{i}, e_{j})Ae_{i}, e_{j}> & = & R^{k}_{iji}<Ae_{k}, e_{j}>-<R^{l}_{ijk}A^{k}_{i}e_{l}, e_{j}>\\
                                                              & = & R^{k}_{iji}A^{j}_{k}-R^{j}_{ijk}A^{k}_{i}\\
                                                              & = & R_{ijik}A^{j}_{k}-R_{ijkj}A^{k}_{i}\\
                                                              & = & R_{jijk}A^{i}_{k}-R_{ijkj}A^{k}_{i}\\
                                                              & = & R_{ijkj}(A^{i}_{k}-A^{k}_{i})\\
                                                              & = & 0.\\
 \end{eqnarray*}

Thus we get

\begin{theorem}
$Trace \Delta A+\Delta Trace(A)=0$.
\end{theorem}

\begin{corollary}
1) If $M$ is compact, then $\int_{M} Trace\Delta A=0$.

2) For any almost complex structure, we have $Trace \Delta J=0$.
\end{corollary}

\bibliographystyle{amsplain}

\end{document}